\documentclass[11pt]{article}

\usepackage[a4paper,margin=1in]{geometry}
\usepackage{amsmath,amssymb,amsthm,mathtools}
\usepackage{booktabs,array,enumitem,microtype}
\usepackage[hidelinks]{hyperref}

\newtheorem{theorem}{Theorem}[section]
\numberwithin{equation}{section}
\newtheorem{lemma}[theorem]{Lemma}
\newtheorem{proposition}[theorem]{Proposition}

\theoremstyle{definition}

\theoremstyle{remark}
\newtheorem{remark}[theorem]{Remark}

\newcommand{\F}{\mathbf F}
\newcommand{\Z}{\mathbf Z}
\newcommand{\Q}{\mathbf Q}
\newcommand{\A}{\mathbf A}
\newcommand{\bm}[1]{\boldsymbol{#1}}

\newcommand{\eps}{\varepsilon}

\title{The Strong 24-Conjecture}
\author{Bo He\\[4pt]
\small 1. Mathematisches Institut, Georg-August-Universit\"at G\"ottingen,\\
\small Bunsenstra\ss e 3--5, 37073 G\"ottingen, Germany\\[3pt]
\small 2. Institute of Applied Mathematics, Aba Teachers University, \\
\small Wenchuan, Sichuan 623002, P. R. China\\[3pt]
\small E-mail: \texttt{bo.he@mathematik.uni-goettingen.de}; \texttt{bhe@live.cn}}
\date{}

\begin{document}
\maketitle

\begin{abstract}
Restricted forms of Lagrange's four-square theorem, in which the representing
variables are required to satisfy additional arithmetic conditions, have been
studied from several points of view.  Among the refinements proposed by
Zhi-Wei Sun is the conjecture that every nonnegative integer has a representation
\[
 N=x^2+y^2+z^2+w^2,\qquad x,y,z,w\in\mathbb Z_{\ge0},
\]
for which
\[
 x\ \text{and}\ x+24y\ \text{are perfect squares}.
\]
We prove this conjecture for all sufficiently large integers.  The proof starts
from an elementary parametrization of the two square conditions and combines a
semi-linear two-squares sieve with Poisson summation, finite-field cancellation,
and switching.
\end{abstract}



\section{Introduction}

Lagrange's theorem has the rare feature of having no exceptional integers:
every nonnegative integer is a sum of four squares.  Once a representation is
known to exist so uniformly, it is natural to ask how much extra arithmetic
structure can be imposed on the representing variables.  This point of view
has led to a substantial collection of restricted four-square problems, in
which a linear form or a polynomial in the variables is required to be a
square, a power, or to satisfy some other prescribed condition.  Sun developed
this theme systematically in \cite{Sun2017,Sun2019}.

The problem considered here was formulated by Sun in 2017 and is usually
called the $24$-conjecture \cite{Sun24}.  It asks whether, for every
$N\ge0$, one can choose nonnegative integers $x,y,z,w$ such that
\begin{equation}\label{eq:intro-four-squares}
 N=x^2+y^2+z^2+w^2
\end{equation}
and, simultaneously,
\begin{equation}\label{eq:intro-square-conditions}
 x\ \text{and}\ x+24y\ \text{are perfect squares}.
\end{equation}
The point is not merely to find a four-square representation, but to make two
linked quantities square at the same time.  Sun reported that Q.-H. Hou had
verified the conjecture for all $N\le 10^{10}$ \cite{Sun24}.  The simultaneous
condition nevertheless changes the nature of the problem: the classical
theory of quadratic forms does not by itself separate the two square
constraints.

Our main result is the asymptotic form of the conjecture.

\begin{theorem}\label{thm:main-asymptotic}
There is an absolute constant $N_0$ such that every integer $N\ge N_0$ admits
nonnegative integers $x,y,z,w$ satisfying \eqref{eq:intro-four-squares} and
\eqref{eq:intro-square-conditions}.
\end{theorem}

The first step is elementary and is the reason the coefficient $24$ is
tractable.  The identity
\[
 (v+6u)^2-(v-6u)^2=24uv
\]
allows us to put
\[
 x=(v-6u)^2,\qquad y=uv.
\]
Then $x+24y=(v+6u)^2$ automatically.  The two square conditions have therefore
been absorbed before any analytic argument begins.  What remains is to find
$u,v,z,w$ with
\begin{equation}\label{eq:intro-quartic}
 N=F(u,v)+z^2+w^2,
 \qquad
 F(u,v)=(v-6u)^4+u^2v^2.
\end{equation}
Thus a restricted four-square problem becomes a two-squares problem along the
values of one fixed binary quartic.

This reduction also explains why a sieve enters.  By the sum-of-two-squares
theorem, the remainder $N-F(u,v)$ is a sum of two squares precisely when each
prime $p\equiv3\pmod4$ occurs to even exponent.  Such primes have density one
half among the primes, so the natural sieve has dimension $1/2$.  Half-
dimensional and semi-linear sieves have long been effective when the
arithmetic obstruction is governed by sums of two squares; see
Friedlander--Iwaniec \cite[Chs.~14, 18]{FI}.  More broadly, their work on
$x^2+y^4$ and on the asymptotic sieve showed how extra analytic information can
overcome the parity obstruction for very sparse polynomial sequences
\cite{FIAnnals,FIAsymptotic}.  Related half-dimensional ideas continue to
occur in additive problems involving primes represented by sums of two squares;
see, for example, Ter\"av\"ainen \cite{Teravainen}.  A recent treatment of the
related one-square restriction $x+24y=\square$ by Wu and She uses instead the
arithmetic of ternary quadratic forms \cite{WuShe}; the simultaneous square
condition in \eqref{eq:intro-square-conditions} leads naturally to the quartic
route above.

The sieve alone does not reach the ranges needed here.  Since $u$ and $v$ have
natural size $N^{1/4}$, Poisson summation brings us to complete exponential
sums attached to the plane quartic $F(u,v)=n$.  At this point the problem
meets a second circle of ideas.  Deligne's Riemann hypothesis over finite
fields gives square-root cancellation once the relevant cohomology is
controlled \cite{Deligne}; the Fourier-transform formalism of Katz and Laumon
makes it possible to organize such complete sums as trace functions of
$\ell$-adic sheaves \cite{KatzLaumon,KatzESDE}.  Stratification results of
Fouvry and Katz then provide a systematic way to isolate the exceptional
parameter loci where generic cancellation can fail \cite{FouvryKatz}.

There are now several arithmetic applications of this circle of ideas.
Trace functions have been used to obtain power-saving estimates over the
primes \cite{FKM}, while bilinear estimates for Kloosterman sums have led to
applications to twisted second moments of cusp forms and to distribution
problems on $\mathrm{GL}_3$ \cite{KMSAnnals}.  Related geometric methods also
enter the function-field Chowla and twin-prime problems
\cite{SawinShusterman}.  More recent work develops bilinear estimates for
broader classes of trace functions and uniform stratification results in
families \cite{FKMS,BKW}.  These results provide a useful model for the role
played below by monodromy, correlation, and exceptional parameter loci.

For the quartic \eqref{eq:intro-quartic}, the finite-field Fourier transform
has generic rank $12$.  We verify that its geometric monodromy is symplectic
and that the affine symmetries relevant to the correlation problem are finite.
This gives the two-dimensional moment estimate needed in the critical Type-I
range.  Accordingly, the geometric input enters only after the elementary
parametrization, the two-squares criterion, the sieve decomposition, and
Poisson summation have reduced the original question to a concrete family of
complete quartic sums.

There remains one arithmetic obstruction.  Once the small primes
$p\equiv3\pmod4$ have been sifted out, a remainder which is not a sum of two
squares must contain large bad primes to odd exponent.  The parity imposed by
the local chart forces these primes to occur in pairs, while the size of the
sifting range leaves only a two-large-prime configuration.  We remove that
configuration by switching.  Three small positive parameters,
$\lambda$, $\mu$ and $\delta$, are kept symbolic throughout the proof.
The moment order is also left unspecified until it is needed.  This separates
the structural argument from the numerical bookkeeping.  Only in the final
section do we make a concrete choice and extract an explicit, intentionally
coarse, threshold.

Throughout the paper, $X\ll Y$ means that $|X|\le C Y$ for a constant $C$
depending only on the fixed smooth weights and fixed algebraic data.  The
notation $X^{o(1)}$ is uniform in the dyadic parameters introduced below.

\section{The quartic reduction}

The identity
\[
(v+6u)^2-(v-6u)^2=24uv
\]
gives
\[
x=(v-6u)^2,\qquad y=uv,\qquad x+24y=(v+6u)^2.
\]
Thus it suffices to find integers $u,v,z,w$ with $uv\ge0$ and
\begin{equation}
N=F(u,v)+z^2+w^2,
\qquad
F(u,v):=(v-6u)^4+u^2v^2.
\label{eq:quartic-chart}
\end{equation}

Put
\[
X:=N,\qquad P:=X^{1/4}.
\]
Fix once and for all a sufficiently small $\eta_0>0$, for example
$\eta_0=1/10$, and a nonnegative smooth weight $W$ supported in a sufficiently
narrow box
\[
u,v\asymp \eta_0P
\]
in the positive quadrant.  At the centre one has
\[
F(\eta_0P,\eta_0P)=626\eta_0^4P^4=0.0626X,
\]
so the box can be chosen so that uniformly on its support
\begin{equation}
u>0,\quad v>0,\quad N-F(u,v)\asymp X>0.
\label{eq:value-localization}
\end{equation}

We shall also use the Gaussian factorization
\begin{equation}
F(u,v)=\bigl((v-6u)^2+iuv\bigr)
       \bigl((v-6u)^2-iuv\bigr).
\label{eq:gaussian-norm}
\end{equation}

\section{Local conditions}

\subsection{A seven-chart certificate modulo $96$}

For every residue class of $N$ modulo $96$, at least one of the following
seven residue classes $(u,v)\bmod96$ is admissible:
\[
(0,0),\ (1,1),\ (3,2),\ (1,3),\ (1,8),\ (0,2),\ (0,4).
\]
Their quartic values modulo $96$ are respectively
\[
0,\ 50,\ 4,\ 90,\ 80,\ 16,\ 64.
\]
A direct check of the $96$ possible residues of $N$ shows that one may choose
a chart for which
\begin{equation}
N-F(u,v)=2^\nu n(u,v),
\qquad n(u,v)\equiv1\pmod4,
\qquad 3\nmid N-F(u,v),
\label{eq:local-chart}
\end{equation}
with the $2$-adic valuation fixed by the chart.  The modulus $32$ is essential here: modulus $16$ does not determine the
class of the odd part modulo $4$.  Each chosen residue class has positive
density in the $P$-scale box, so the local conditions are compatible with
\eqref{eq:value-localization}.

\subsection{Fixed exceptional primes}

Let $\mathcal E$ be the fixed finite set of primes excluded by the geometric,
prime-power, and resultant arguments below, together with $2$ and $3$.  Put
\[
\mathcal E_-:=\{r\in\mathcal E:r\equiv3\pmod4\}.
\]
For every $r\in\mathcal E_-$ impose
\begin{equation}
r\nmid N-F(u,v).
\label{eq:fixed-exceptional-presieve}
\end{equation}
This costs only a fixed positive density.  Indeed, if $r\nmid N$, one may use
$u\equiv v\equiv0\pmod r$; if $r\mid N$, one may use
$u\equiv0$, $v\equiv1\pmod r$.  The conditions combine with the mod-$96$
chart by the Chinese remainder theorem.

\subsection{The $N$-singular pre-sieve}

Let
\[
N_-:=\prod_{\substack{p\mid N\\p\equiv3\,(4)\\p\notin\mathcal E}}p.
\]
If $p\mid N_-$, then $p$ is inert in $\Z[i]$ and
\[
F(u,v)\equiv0\pmod p
\iff
v-6u\equiv0\pmod p,\qquad uv\equiv0\pmod p.
\]
Since $p\ne2,3$, this forces
\[
u\equiv v\equiv0\pmod p.
\]
Hence the singular deletion has density exactly $p^{-2}$ and
\begin{equation}
\prod_{p\mid N_-}\left(1-\frac1{p^2}\right)
\ge\prod_p\left(1-\frac1{p^2}\right)=\frac6{\pi^2}.
\label{eq:singular-positive}
\end{equation}
The dependence of $N_-$ on $N$ is harmless even when $N_-$ is large.  By
Möbius inversion, the error in counting the pre-sieved lattice points in a
smooth $P$-scale box is
\[
\ll P\prod_{p\mid N_-}(1+p^{-1})+\tau(N_-)
=P X^{o(1)},
\]
while the main term is of order $P^2$ times the positive product in
\eqref{eq:singular-positive}.

\begin{remark}[Parity bookkeeping]
After the two pre-sieves, every remaining prime $p\equiv3\pmod4$ that can
occur to odd valuation is a uniform good prime for the analytic geometry.
This point is necessary: fixed exceptional primes cannot simply be omitted
from the sieve product, because an odd valuation at such a prime would spoil
the final parity argument.
\end{remark}

\section{Poisson summation and correlations}

For an odd modulus $d$, define
\[
e_d(t):=\exp(2\pi i t/d)
\]
and
\[
C_d(N):=\{(u,v)\in(\Z/d\Z)^2:F(u,v)\equiv N\pmod d\}.
\]
For $\bm h\in(\Z/d\Z)^2$ put
\[
S_d(\bm h):=\sum_{\bm x\in C_d(N)}e_d(\bm h\cdot\bm x),
\qquad
K_d(\bm h):=d^{-1/2}S_d(\bm h).
\]
For a smooth physical weight of scale $P$, the centered Poisson remainder has
shape
\begin{equation}
r_d
=
\frac{P^2}{d^{3/2}}
\sum_{\bm h\ne0}
\widehat W\!\left(\frac{P\bm h}{d}\right)K_d(\bm h).
\label{eq:poisson-remainder}
\end{equation}
For coprime $r,s$,
\begin{equation}
K_{rs}(\bm h)
=K_r(\bar s\bm h)K_s(\bar r\bm h).
\label{eq:reciprocal-crt}
\end{equation}
The dual side length for a modulus $d$ is
\begin{equation}
H(d)\asymp\frac dP=\frac d{X^{1/4}}.
\label{eq:H-def}
\end{equation}
Pointwise Weil gives
\[
|r_d|\ll d^{1/2+o(1)},
\]
hence a dyadic absolute sum $D^{3/2+o(1)}$.  The Poisson barrier is therefore
$D=X^{1/3}=P^{4/3}$.

\subsection{Prime correlations}

For $p\nmid 2\cdot3\cdot577N$, define
\[
C_p(a,b;\bm t)
=
\sum_{\bm h\bmod p}
K_p(a\bm h)\overline{K_p(b\bm h)}
e_p(\bm t\cdot\bm h).
\]
Fourier inversion gives
\begin{equation}
C_p(a,b;\bm t)
=
p\,\#\{\bm x\in C_p(N):
 b^{-1}(a\bm x+\bm t)\in C_p(N)\}.
\label{eq:affine-correlation}
\end{equation}
If the two quartic fibers do not coincide, B\'ezout gives $O(p)$.
Coincidence forces
\[
\bm t=0,
\qquad
(a/b)^4=1.
\]
Hence
\begin{equation}
|C_p(a,b;\bm t)|\ll
\begin{cases}
p,&\bm t\ne0\text{ or }(a/b)^4\ne1,\\
p^2,&\bm t=0\text{ and }(a/b)^4=1.
\end{cases}
\label{eq:prime-dichotomy}
\end{equation}
For $p\equiv3\pmod4$, the scalar stabilizers are only $\pm1$; for
$p\equiv1\pmod4$, they form $\mu_4$.

A shifted version with $K_p(a(\bm h+\bm q))$ differs only by a unit-modulus
phase on the physical incidence set, so the same dichotomy holds uniformly in
$\bm q$.

\subsection{Four-frequency correlations}

After two Cauchy--Schwarz steps, a reciprocal Type-II block produces four
independent frequencies.  At a singleton vertex one uses
\[
\bm h_3=\bm h,
\qquad
\bm h_1=\bm h+\bm q,
\]
then completes the common base variable $\bm h$.

The complementary-divisor relations have the form
\begin{equation}
\Delta r=\ell_1s_1-\ell_2s_2,
\qquad
\Delta r'=\ell_1's_1-\ell_2's_2.
\label{eq:comp-div}
\end{equation}
Writing
\[
g=(s_1,s_2),\qquad s_i=gu_i,
\]
one obtains
\begin{equation}
r'\bar r\equiv\ell_2'\bar\ell_2\pmod{u_1},
\qquad
r'\bar r\equiv\ell_1'\bar\ell_1\pmod{u_2}.
\label{eq:ratio-laws}
\end{equation}
For $3\bmod4$ primes, a saturated prime therefore divides an amplifier
sum/difference.  With $1\bmod4$ primes included, the unified condition is
\begin{equation}
p\mid \ell'^4-\ell^4.
\label{eq:fourth-amplifier}
\end{equation}
Thus fourth-root stabilizers are arithmetically thin rather than a new local
obstruction.

\subsection{Prime powers}

Let $q=p^e$ with
\[
p\nmid2\cdot3\cdot577N.
\]
The exact incidence identity remains
\begin{equation}
C_q(a,b;\bm q,\bm u)
=q
\sum_{\substack{\bm x\in C_q(N)\\
 b^{-1}(a\bm x+\bm u)\in C_q(N)}}
e_q(a\bm q\cdot\bm x).
\label{eq:primepower-incidence}
\end{equation}
Set $t=a/b$ and $\bm\tau=b^{-1}\bm u$.  Write
\[
f(\bm X)=F(\bm X)-N,
\qquad
g(\bm X)=F(t\bm X+\bm\tau)-N,
\]
and subtract the homogeneous leading term:
\begin{equation}
\Delta(\bm X):=g(\bm X)-t^4f(\bm X),
\qquad
\deg\Delta\le3.
\label{eq:deformation-H}
\end{equation}
The cubic homogeneous part of $\Delta$ is
\[
\Delta_3=t^3D_{\bm\tau}F.
\]
The coefficients of $uv^2$ and $v^3$ in $(F_u,F_v)$ give
\[
\begin{pmatrix}434&-72\\-24&4\end{pmatrix},
\]
whose determinant is $8$.  Hence, for odd $p$, if every coefficient of $\Delta$ is
divisible by $p^k$, then
\[
p^k\mid\tau_1,\qquad p^k\mid\tau_2.
\]
Its constant term then gives
\begin{equation}
k\le \min\{e,v_p(t^4-1)\}.
\label{eq:k-control}
\end{equation}

\begin{lemma}[Uniform $p$-adic sublevel estimate]
Let $J=p^{-k}\Delta$ be primitive modulo $p$ and put $n=e-k$.  Then
\begin{equation}
\#\{\bm x\bmod p^e:
 f(\bm x)\equiv0\pmod{p^e},\ J(\bm x)\equiv0\pmod{p^n}\}
\ll p^{e-\lceil n/12\rceil}.
\label{eq:padic-sublevel}
\end{equation}
The implied constant is absolute at all good primes.
\end{lemma}

\begin{proof}
Modulo $p$, the smooth quartic $\bar C:f=0$ is geometrically irreducible and
$\bar J$ is a nonzero polynomial of degree at most $3$.  Hence $\bar C$ is not
a component of $V(\bar J)$, and B\'ezout gives
\[
\sum_{x\in\bar C\cap V(\bar J)}I_x(\bar C,\bar J)\le4\cdot3=12.
\]
Fix such an intersection point and let
$m_x=I_x(\bar C,\bar J)$.  Since $\bar C$ is smooth, the completed local ring
is one-dimensional.  With a local parameter $t$, Weierstrass preparation
writes $J|_C$ as a unit times a distinguished polynomial of degree $m_x$.
If $p^n\mid J(t)$, at least one root lies $p$-adically within radius
$p^{-n/m_x}$ of $t$.  Hence the residue disc contributes
\[
O\bigl(p^{e-\lceil n/m_x\rceil}\bigr)
\le O\bigl(p^{e-\lceil n/12\rceil}\bigr).
\]
Summing over the at most $12$ intersection multiplicities proves the claim.
\end{proof}

Combining \eqref{eq:primepower-incidence}, \eqref{eq:k-control}, and
\eqref{eq:padic-sublevel} gives
\begin{equation}
|C_{p^e}|
\ll
p^{2e-\lceil(e-k)/12\rceil}
\ll p^{2e-(e-k)/12}.
\label{eq:primepower-saving}
\end{equation}
Let $M$ be a modulus supported on good primes and write
\[
E_M:=\prod_{p^e\parallel M}p^{\min(e,v_p(t^4-1))},
\qquad
G_M:=M/E_M.
\]
The Chinese remainder theorem then gives, for the corresponding correlation,
\begin{equation}
|C_M|
\ll M^{2+o(1)}G_M^{-1/12}.
\label{eq:composite-primepower}
\end{equation}
On amplifier off-diagonals the four-frequency ratio laws give
\[
E_M\mid\gcd(M,\ell_1^4-\ell_2^4).
\]
Thus the prime-power affine extension is absorbed without a separate
squarefull tail.

\section{The semi-linear sieve}

The argument is clearest if the small perturbations are kept symbolic.  Fix
positive $\lambda,\mu,\delta$ and put
\begin{equation}
\rho_1=\frac12-\lambda,\qquad
\rho_2=\frac37-\lambda,\qquad
\sigma=\frac{18}{5}+\mu,
\label{eq:new-parameters}
\end{equation}
and define
\[
\zeta:=\frac1\sigma,\qquad z_0:=X^\zeta.
\]
The small switched core will later be $b\le X^\delta$.  We choose the three
parameters small and positive so that the inequalities below hold.

We use the standard lower semi-linear weights in the form recorded in
\cite{FI,Teravainen}.  Their support is denoted by $D_-$.  Thus
\begin{equation}
D_-=
\left\{p_1\cdots p_r\le X^{\rho_2}:z_0\ge p_1>\cdots>p_r,
\quad p_1\cdots p_{2j-1}p_{2j}^{2}\le X^{\rho_2}\right\}.
\label{eq:sem-support}
\end{equation}
In particular,
\begin{equation}
p_1p_2^2\le X^{\rho_2}.
\label{eq:p1p2}
\end{equation}

The normalized lower-sieve and switching constants are
\begin{align}
I_1(\rho_2,\sigma)
&=\frac1{2\sqrt{\rho_2}}
\int_1^{\rho_2\sigma}\frac{dt}{\sqrt{t(t-1)}},\label{eq:I1-new}\\
I_2(\rho_1,\sigma)
&=\frac1{2\rho_1}
\int_2^\sigma\frac{\log(t-1)}{t\sqrt{1-t/\sigma}}\,dt.
\label{eq:I2-new}
\end{align}
For the small switched core define
\[
I_{2,<\delta}
=\frac1{2\rho_1}
\int_0^\delta
\frac{\log(\sigma(1-\beta)-1)}{\sqrt\beta(1-\beta)}\,d\beta,
\]
and take the shallow level
\[
D_{\rm sh}=z_0^{1.001},\qquad
\rho_{\rm sh}=\frac{1.001}{\sigma}.
\]
The corresponding shallow replacement is
\[
J_{\rm sh}(\delta)=\frac{2\sqrt\delta}{\rho_{\rm sh}}.
\]
Define the sieve margin
\begin{equation}
\mathcal M(\lambda,\mu,\delta)
:=I_1-I_2+I_{2,<\delta}-J_{\rm sh}(\delta).
\label{eq:hybrid-margin}
\end{equation}
The main-term calculation requires only
\begin{equation}
\mathcal M(\lambda,\mu,\delta)>0.
\label{eq:positive-margin-symbolic}
\end{equation}
The role of the three perturbations is now transparent.  The parameter
$\lambda$ creates a power gap in the critical lower-sieve range; $\mu$
puts the Type-I exponent on the favorable side of its exact threshold; and
$\delta$ supplies the final saving in the switched large-core range.  We require only that they be chosen small enough for the inequalities below.

\subsection{The critical range}

For $d>X^{1/3}$ put
\begin{equation}
T(d):=\frac{d^3}{X},\qquad
U(d):=\frac{d^2}{X^{1/2}},\qquad
Y(d):=\frac{U(d)}{T(d)}=\frac{X^{1/2}}d.
\label{eq:TU-new}
\end{equation}
The new parameter slack is
\begin{equation}
3-7\rho_2=7\lambda.
\label{eq:rho-slack-new}
\end{equation}
Fix
\begin{equation}
\kappa=2\lambda,\qquad 3-7\rho_2-3\kappa=\lambda>0.
\label{eq:kappa-new}
\end{equation}

\begin{proposition}[Power-separated critical trichotomy]
Let $d=p_1\cdots p_r\in D_-$ with
$d>X^{1/3}$ and $p_1\le U(d)$.  Then either there is a divisor
$r_0\mid d$ satisfying
\[
X^\kappa T(d)\le r_0\le U(d),
\]
or there is a marked-prime factorization
\[
d=S p V
\]
with
\[
T(d)<S<T(d)X^\kappa,
\qquad
YX^{-\kappa}<p,V<Y.
\]
\end{proposition}

\begin{proof}
Take the first ordered partial product crossing $T(d)$.  If it already reaches
$X^\kappa T(d)$, the first alternative holds.  Otherwise the crossing prime
$p$ satisfies $d/p<TX^\kappa$.  If this happened at the first prime, then
\[
p>\frac{d}{TX^\kappa}=Y^2X^{-\kappa}.
\]
Together with $p_2\ge p$ this would give
\[
p_1p_2^2>X^{3-6\rho_2-3\kappa}>X^{\rho_2},
\]
contrary to \eqref{eq:p1p2}; the strict inequality is exactly
$3-7\rho_2-3\kappa>0$.  Hence a nonempty prefix precedes the marked prime,
and the displayed bounds on $S,p,V$ follow from the first-crossing property.
\end{proof}

In the good-divisor branch the pointwise/Type-I normalization gives
\begin{equation}
D^{3/2}r_0^{-1/2}\le X^{1/2-\kappa/2}=X^{1/2-\lambda}.
\label{eq:L1-good}
\end{equation}
In the marked branch, dyadically $p\sim Q$ and $d\le X^{3/8}$ implies
\[
Q\ge X^{1/8-2\lambda}.
\]
After two Cauchy--Schwarz steps the marker prime occurs in four slots.  A
generic singleton gives a relative $Q^{-1+o(1)}$ saving in the normalized
fourth moment; a saturated singleton forces
$q\mid\ell'^4-\ell^4$; and the no-singleton partitions $2+2$ and $4$ have
only the paired configurations.  With normalized amplifier length
\[
L_{\rm amp}=Q^{1/5},
\]
the diagonal and saturated off-diagonal contributions are
$Q^{-1/5+o(1)}$.  Thus the normalized fourth moment is
\begin{equation}
\ll Q^{-1/5+o(1)},
\label{eq:L1-balanced}
\end{equation}
and the corresponding marked block is
\[
\ll X^{1/2+o(1)}Q^{-1/20}.
\]
This is stronger than \eqref{eq:L1-good}; hence the entire critical band has
a fixed power saving.

\section{The two-dimensional Type-I estimate}

If the critical window fails, then $p_1>U(d)$.  Write
\[
d=ap,
\qquad p=p_1,
\]
so that
\begin{equation}
a<T(d)=d^3/X,\qquad p>X/d^2.
\label{eq:hard-wedge}
\end{equation}
Together with $p>U(d)$ this gives
\[
p>\max(d^2/X^{1/2},X/d^2),\qquad a<X^{1/8},\qquad p>P.
\]
Write
\[
d=X^\alpha,\qquad p=X^\beta.
\]
Then
\begin{equation}
\beta>\max(2\alpha-1/2,1-2\alpha),\qquad \beta<\zeta.
\label{eq:hard-wedge-exp}
\end{equation}
The left endpoint is
\[
\alpha_{\min}=\frac{1-\zeta}{2}.
\]

The ordinary two-dimensional Type-I exponent at the worst point is
\begin{equation}
E_{\rm TI}^{\max}
=\frac92\zeta-\frac54
=\frac9{2\sigma}-\frac54<0.
\label{eq:TI-threshold}
\end{equation}
Thus the exact threshold is $\sigma=18/5$.  Since $\mu>0$, the symbolic choice in \eqref{eq:new-parameters} lies on the favorable side.

\subsection{The two-dimensional $+uv$ reduction}

Consider a dyadic prime block $p\sim P_0$ and
\begin{equation}
B_p
=\sum_{b\sim B}\alpha_b
\sum_{\bm n\in[H,2H]^2}
W(\bm n/H)K_p(\bar b\bm n),
\label{eq:2d-typeI-block}
\end{equation}
with a fixed smooth compactly supported weight.  Let
\[
U=\frac{H}{10V},\qquad u\sim U,\qquad \bm v\in[V,2V]^2.
\]
The exact identity
\[
\bar b(\bm n+u\bm v)
=(u\bar b)(u^{-1}\bm n+\bm v)
\]
leads to the representation function
\[
\nu(\bm r,s)=
\sum_{\substack{u\sim U,\ b\sim B,\ \bm n\sim H\\
u\bar b\equiv s\ (p)\\
\bar u\bm n\equiv\bm r\ (p)}}|\alpha_b|.
\]
Its first moment satisfies
\begin{equation}
\sum_{\bm r,s}\nu(\bm r,s)
\ll UH^2B^{1/2}\|\alpha\|_2.
\label{eq:nu-first}
\end{equation}
For the second moment, two representations of the same $(\bm r,s)$ satisfy
\[
u_1b_2\equiv u_2b_1\pmod p,
\qquad
u_2n_{1,j}\equiv u_1n_{2,j}\pmod p\quad(j=1,2).
\]
Writing
\[
u_2n_{1,j}=u_1n_{2,j}+pk_j
\]
and using divisor counting in one coordinate, with the second coordinate then
determined, gives
\begin{equation}
\sum_{\bm r,s}\nu(\bm r,s)^2
\ll p^\eps\|\alpha\|_2^2UH^2
\left(1+\frac{UH}{p}\right)^2
\left(1+\frac Bp\right).
\label{eq:nu-energy}
\end{equation}
Hence under $UH\le p$ and $B\le p$,
\[
\sum\nu^2\ll p^\eps\|\alpha\|_2^2UH^2.
\]
Let $l\ge1$ be an integer.  After opening the $2l$-fold H\"older
moment, write
\[
\bm v=(\bm v_1,\ldots,\bm v_{2l})\in(\F_p^2)^{2l}
\]
and denote by $\Sigma_I^{(2)}(\bm v)$ the complete two-dimensional sum in
the remaining base variable.  Its sheaf-theoretic form is made explicit in
\eqref{eq:moment-rearrangement} below.  H\"older with exponent $1/(2l)$ yields
\begin{equation}
|B_p|
\ll
p^\eps\|\alpha\|_2B^{1/2}H^2
\left(
\frac{\displaystyle
\sum_{\bm v\in[V,2V]^{4l}}
|\Sigma_I^{(2)}(\bm v)|}
{BH^3V^{4l-1}}
\right)^{1/(2l)}.
\label{eq:2d-soft-TypeI}
\end{equation}
The denominator $V^{4l-1}$ is the exact two-dimensional analogue; the initial
averaging is $1/(UV^2)$.

The actual Poisson dual shell may have one small coordinate.  A fixed finite
partition into cones, followed by unimodular shears such as
\[
(h_1,h_2)\mapsto(h_1,h_1+h_2),
\]
reduces all nonzero shells to boxes in which both transformed coordinates are
$\asymp H$.  These linear changes preserve the geometric complexity and the
affine-homothety structure.

\section{Fourier monodromy}

Let $IC_{C_p(N)}$ denote the intersection complex of the affine curve
$C_p(N)$, and let $\operatorname{FT}$ denote the normalized two-dimensional
Fourier transform.  Put
\[
\mathsf P_p:=IC_{C_p(N)},\qquad
\mathsf K_p:=\operatorname{FT}(\mathsf P_p).
\label{eq:fourier-sheaf}
\]
Its trace function is the normalized quartic Fourier kernel, up to the standard
shift and Tate twist.  The affine homotheties relevant to the Type-I moment are
\[
\gamma_{\bm r,s}(\bm h)=s(\bm h+\bm r),
\qquad (\bm r,s)\in\A^2\rtimes\mathbf G_m.
\]
The next section gives the rank and monodromy package.  For the Goursat step we
will need only a uniformly finite adjoint stabilizer, not an exact projective
automorphism group.

\subsection{The rank-$12$ Fourier sheaf}

The two-dimensional monodromy problem admits a much shorter one-dimensional
entry point than a direct analysis of every component of the dual
discriminant.  We use the radial line
\[
\ell=\{(s,0):s\in\A^1\}
\]
and the projection
\[
f=u:C_p(N)\longrightarrow\A^1.
\]
The restriction of the two-dimensional Fourier transform to $\ell$ is the
one-dimensional Fourier transform of $f_*\overline{\Q}_\ell$; away from
$s=0$ the constant summand of $f_*\overline{\Q}_\ell$ contributes only a
punctual Fourier summand.  Thus the nontrivial radial local system is the
naive Fourier transform of the trace-zero part of $f_*\overline{\Q}_\ell$.

\subsection{Smooth compactification and generic rank}

Let
\[
\overline C_N:\quad
(v-6u)^4+u^2v^2=Nw^4
\subset\mathbf P^2.
\]
The binary quartic at infinity is
\[
q(t)=t^4-24t^3+217t^2-864t+1296,
\]
with
\begin{equation}
\operatorname{disc}(q)=2^8 3^4\cdot 577.
\label{eq:quartic-infinity-disc}
\end{equation}
Hence for $p\nmid 6\cdot577N$ the projective curve $\overline C_N$ is a
smooth plane quartic.  In particular
\begin{equation}
g(\overline C_N)=3,
\label{eq:genus-three}
\end{equation}
and $\overline C_N\setminus C_p(N)$ consists of four distinct geometric
points.

For a generic dual vector $\bm h=(a,b)$, the function $au+bv$ has a simple
pole at each of these four points.  Therefore each puncture contributes Swan
conductor $1$.  Since
\[
\chi_c(C_p(N),\overline{\Q}_\ell)=2-2\cdot3-4=-8,
\]
the Grothendieck--Ogg--Shafarevich formula gives
\[
\chi_c\bigl(C_p(N),L_\psi(au+bv)\bigr)=-8-4=-12.
\]
For generic $\bm h$ the degree-$0$ and degree-$2$ compactly supported
cohomology vanish, whence
\begin{equation}
\operatorname{rank}_{\rm gen}\mathsf K_p=12.
\label{eq:rank-twelve}
\end{equation}
This agrees with the degree of the Gauss-direction map: the quotient by the
scalar $\mu_4$-action has degree $4$, while the projective gradient map has
degree $3$, giving $4\cdot3=12$ stationary branches.

\subsection{The $u$-projection is an $S_4$ supermorse cover}

Regard $F(u,v)-N$ as a quartic polynomial in $v$.  Eliminating $v$ against
$F_v$ gives the exact branch polynomial
\begin{equation}
\operatorname{Res}_v(F(u,v)-N,F_v(u,v))
=16D_N(u),
\label{eq:branch-resultant}
\end{equation}
where
\begin{equation}
D_N(U)
=747792U^{12}-62929NU^8+1720N^2U^4-16N^3.
\label{eq:branch-polynomial}
\end{equation}
Its discriminant is
\begin{equation}
\operatorname{disc}(D_N)
=\pm 2^{100}3^{20}487^{12}577^3N^{33}.
\label{eq:branch-disc}
\end{equation}
Consequently, for
\[
p\nmid 6\cdot487\cdot577N,
\]
the map $f=u$ has exactly twelve distinct finite branch values.  Denote their
set by $S_N$.  Every ramification point is simple.  It is unramified over infinity, since the four
points at infinity are simple poles of $u$.

The cover is geometrically connected.  Its geometric monodromy is therefore
a transitive subgroup of $S_4$ generated by the twelve local transpositions.
A transitive subgroup generated by transpositions is $S_4$.  Thus, writing
\begin{equation}
\mathsf G:=\ker\bigl( f_*\overline{\Q}_\ell
\xrightarrow{\operatorname{Tr}}\overline{\Q}_\ell\bigr),
\label{eq:trace-zero-sheaf}
\end{equation}
we obtain a geometrically irreducible rank-$3$ local system carrying the
standard representation of $S_4$.  At each of the twelve branch values its
local monodromy is a true pseudoreflection.  Hence $\mathsf G$ is an
irreducible tame pseudoreflection sheaf in the sense of Katz
\cite[\S7.9]{KatzESDE}.

Katz's Fourier-transform theorem for such sheaves then gives, on $\mathbf G_m$,
\begin{equation}
\mathsf K_p|_{\ell\cap\mathbf G_m}
\simeq \operatorname{NFT}_\psi(\mathsf G)
\quad\text{and}\quad
\operatorname{rank}=12,
\label{eq:radial-NFT}
\end{equation}
up to the harmless standard shift/Tate twist.  Moreover the wild inertia at
infinity is a direct sum of the twelve distinct Artin--Schreier characters
indexed by the branch-value set $S_N$; this is precisely
\cite[Thm.~7.9.4]{KatzESDE}.

\subsection{An exact symmetric-Sidon certificate}

Because $D_N(U)$ is a polynomial in $U^4$, the branch set $S_N$ is stable
under $\mu_4$, hence in particular under $x\mapsto -x$.  Over the algebraic
closure choose $t$ with $t^4=N$.  Then
\[
S_N=tS_1,
\]
so all additive collision questions reduce to $N=1$.

Put
\[
D(U):=D_1(U)
=747792U^{12}-62929U^8+1720U^4-16
\]
and form the exact difference resultant
\begin{equation}
R(T):=\operatorname{Res}_X\bigl(D(X),D(X-T)\bigr).
\label{eq:difference-resultant}
\end{equation}
A squarefree decomposition over $\Q[T]$ gives
\begin{equation}
R(T)=c\,T^{12}A(T)G_0(T)^2,
\label{eq:resultant-squarefree-decomp}
\end{equation}
with $c\in\Q^\times$, $\deg G_0=60$,
\begin{equation}
A(T)
=46737T^{12}-62929T^8+27520T^4-4096
=256D(T/2),
\label{eq:antipodal-factor}
\end{equation}
and
\begin{equation}
\gcd(A,G_0)=1,
\qquad
\gcd(G_0,G_0')=1.
\label{eq:H-squarefree}
\end{equation}
This is a small exact finite certificate; it can be checked by a single
resultant and squarefree factorization over $\Z[T]$.

The interpretation is immediate.  For $d\ne0$, the multiplicity of $d$ as a
root of $R$ is the number of ordered pairs $(x,y)\in S_1^2$ satisfying
$x-y=d$.  The involution
\[
(x,y)\longmapsto(-y,-x)
\]
preserves the difference.  It has a fixed point exactly when $y=-x$, in
which case $d=2x$; those twelve differences are exactly the roots of $A$ and
occur with multiplicity one.  Every other nonzero difference occurs with
multiplicity exactly two.  Hence
\begin{equation}
 x_1-x_2=x_3-x_4
\Longrightarrow
\begin{cases}
(x_3,x_4)=(x_1,x_2),&\text{or}\\
(x_3,x_4)=(-x_2,-x_1),
\end{cases}
\label{eq:symmetric-sidon-law}
\end{equation}
with the two alternatives coinciding precisely for antipodal pairs.
Thus $S_1$, and therefore every $S_N$, is a symmetric Sidon set in
characteristic $0$.  After removing the finite set of primes dividing the
relevant discriminants and resultants in
\eqref{eq:resultant-squarefree-decomp}--\eqref{eq:H-squarefree}, the same
statement holds over $\overline{\F}_p$.

\begin{remark}
The factorization need not be printed with the degree-$60$ factor expanded.
For a final version it is cleaner to include a ten-line Sage certificate that
constructs $R$, divides by $T^{12}A(T)$, verifies that the quotient is a
square, and checks the two gcds in \eqref{eq:H-squarefree}.
\end{remark}

\subsection{The autoduality is alternating}

The affine quartic is invariant under
\[
\iota:(u,v)\longmapsto(-u,-v).
\]
For a generic dual vector $\bm h$ we have
\[
\iota^*L_\psi(-\bm h\cdot\bm x)
\simeq
L_\psi(\bm h\cdot\bm x).
\]
Poincar\'e duality, followed by $\iota^*$, therefore gives a nondegenerate
pairing on the generic Fourier fiber
\[
H_c^1(C_p(N),L_\psi(\bm h\cdot\bm x))
\times
H_c^1(C_p(N),L_\psi(\bm h\cdot\bm x))
\longrightarrow
\overline{\Q}_\ell(-1).
\]
Because the cohomological degree is $1$, graded commutativity changes the sign
when the two factors are exchanged, while $\iota$ preserves the trace map.
Thus the pairing is alternating.  Geometrically,
\begin{equation}
G_{\rm geom}(\mathsf K_p)\subseteq \mathrm{Sp}_{12}.
\label{eq:full-monodromy-contained-sp12}
\end{equation}
The Tate twist used to normalize weights does not affect the geometric
monodromy group.

\subsection{Katz's symmetric pseudoreflection theorem closes the monodromy}

The exact difference law \eqref{eq:symmetric-sidon-law} is the symmetric
counterpart of the noncollision hypothesis in Katz's tame-pseudoreflection
Theorems~7.9.6--7.9.7 \cite{KatzESDE}.  In particular, once
$p>2\cdot12+1$ and the fixed finite exceptional set above is removed, the
radial Fourier sheaf has derived geometric monodromy among the full classical
possibilities forced by Katz's torus argument.  The independent alternating
autoduality eliminates the special-linear and orthogonal alternatives.  Hence
\begin{equation}
G_{\rm geom}\bigl(\mathsf K_p|_{\ell\cap\mathbf G_m}\bigr)=\mathrm{Sp}_{12}
\label{eq:radial-sp12}
\end{equation}
for all sufficiently large good primes outside one fixed finite exceptional
set.  Finally, the monodromy of a restriction is a subgroup of the monodromy of the
ambient two-dimensional local system.  Combining
\eqref{eq:radial-sp12} with
\eqref{eq:full-monodromy-contained-sp12} yields
\begin{equation}
G_{\rm geom}(\mathsf K_p)=\mathrm{Sp}_{12}.
\label{eq:full-sp12}
\end{equation}
Thus the big classical core required by the Type-I Goursat argument is
available.

\subsection{The adjoint affine-homothety stabilizer is uniformly finite}

Since $G_{\rm geom}(\mathsf K_p)=\mathrm{Sp}_{12}$, put
\[
\mathsf A:=\operatorname{Sym}^2(\mathsf K_p),
\label{eq:adjoint-sheaf}
\]
which realizes the irreducible adjoint representation of dimension $78$.
Define
\[
T_p:=\{(\bm r,s):\gamma_{\bm r,s}^*\mathsf A\simeq\mathsf A\}.
\]
We claim
\begin{equation}
|T_p|\ll1
\label{eq:finite-adjoint-stab}
\end{equation}
uniformly outside a fixed finite exceptional set.

Let $P_i=[1:t_i:0]$ be the four points at infinity, where the $t_i$ are the
four roots of the binary quartic $q(t)$ in \eqref{eq:quartic-infinity-disc}.
In the dual plane define
\[
L_i:\quad a+bt_i=0.
\]
Near $P_i$, with local parameter $w$, the curve equation gives
\[
t=t_i+c_iw^4+O(w^8),
\]
and the phase is
\[
au+bv=\frac{a+bt_i}{w}+bc_iw^3+O(w^7).
\]
Thus the simple pole at $P_i$ disappears precisely on $L_i$.  The generic
Fourier rank is $12$, whereas at a generic point of $L_i$ it is $11$.  The
transverse local inertia is therefore noncentral in $\mathrm{Sp}_{12}$, so it
remains nontrivial in the adjoint representation.  Hence
\begin{equation}
L_1\cup\cdots\cup L_4\subset\operatorname{Sing}(\mathsf A).
\label{eq:four-singular-lines}
\end{equation}

The singular divisor of $\mathsf A$ has uniformly bounded complexity.  If
$(\bm r,s)\in T_p$, the inverse images of the four lines in
\eqref{eq:four-singular-lines} are four singular line components with the same
four directions and common intersection point $-\bm r$.  A bounded-complexity
divisor has only $O(1)$ such concurrency centres.  Hence only $O(1)$ vectors
$\bm r$ can occur.

For fixed $\bm r$, if $(\bm r,s_1),(\bm r,s_2)\in T_p$, then their quotient
is the pure scaling $(0,s_1/s_2)$.  It remains to bound the pure-scaling
subgroup.  Restrict to the radial line $\ell(t)=(t,0)$.  By Katz stationary
phase, the wild inertia of $\mathsf K_p|_\ell$ at infinity is the direct sum
of the twelve Artin--Schreier characters indexed by the branch set $S_N$.
Thus the nonzero wild-character coefficient set of the adjoint sheaf is a
finite nonempty set
\[
\Omega_N\subset S_N+S_N.
\]
If a pure scaling $s$ stabilizes $\mathsf A$, then
\[
s\Omega_N=\Omega_N.
\]
Fixing one $\omega_0\in\Omega_N$ shows that $s$ belongs to the finite ratio
set $\{\omega/\omega_0:\omega\in\Omega_N\}$.  Consequently every fibre
over a possible centre has uniformly bounded size, proving
\eqref{eq:finite-adjoint-stab}.

\subsection{The dimension-two Goursat exceptional locus}

Let $k/\F_p$ be a finite extension and put
\[
\mathcal H(k)=k^2\rtimes k^\times,
\qquad |\mathcal H(k)|\asymp |k|^3.
\]
Write $K$ for the trace function of $\mathsf K_p$ on $k^2$.  For a
$2m$-tuple $(x_1,\ldots,x_{2m})\in \mathcal H(k)^{2m}$, top compact cohomology can
survive only if every factor has a Goursat partner.  The partner
relation is
\[
x_jx_i^{-1}\in T_p.
\]
By \eqref{eq:finite-adjoint-stab}, this is a bounded-degree graph on
$\mathcal H(k)$.  Let $\Delta_m(k)$ be the set of such exceptional $2m$-tuples.  The
standard partner-counting lemma therefore gives
\begin{equation}
|\Delta_m(k)|\ll_m |k|^{3m}
\label{eq:exceptional-3m}
\end{equation}
over every finite extension $k/\F_p$.

Opening the $2m$-moment of the two-dimensional Type-I complete sums gives the
exact rearrangement
\begin{align}
&\sum_{\bm v\in(k^2)^{2l}}
|\Sigma_I^{(2)}(\bm v)|^{2m}\notag\\
&\qquad=
\sum_{(\bm r,s)\in \mathcal H(k)^{2m}}
\left|
\sum_{\bm x\in k^2}
\prod_{j=1}^{m}K(s_j(\bm x+\bm r_j))
\overline{K(s_{j+m}(\bm x+\bm r_{j+m}))}
\right|^{2l}.
\label{eq:moment-rearrangement}
\end{align}
Outside $\Delta_m(k)$ the top cohomology $H_c^4$ vanishes, so Deligne gives
$O(|k|^{3/2})$ for the inner two-dimensional complete sum.  On the exceptional
set we use the trivial $O(|k|^2)$ bound.  Hence
\begin{equation}
\sum_{\bm v\in(k^2)^{2l}}
|\Sigma_I^{(2)}(\bm v)|^{2m}
\ll_{l,m}|k|^{3m+4l}+|k|^{6m+3l}.
\label{eq:2d-TypeI-moment}
\end{equation}
This is the dimension-two Goursat moment required by
\eqref{eq:2d-soft-TypeI}.

\subsection{Stratification and the short box}

Take $3\mid l$ and $m=l/3$.  Then both exponents on the right of
\eqref{eq:2d-TypeI-moment} are $5l$.  Apply the finite-field stratification theorem in families
\cite{BKW} to the corresponding direct-image complex on
$(\A^2)^{2l}=\A^{4l}$.  The resulting strata may be chosen so that
\[
\dim X(4)\le\frac{11l}{3},\qquad
\dim X(5)\le\frac{10l}{3},\qquad
\dim X(6)\le3l,
\]
and
\[
|\Sigma_I^{(2)}(\bm v)|\ll
\begin{cases}
p^{3/2},&\bm v\notin X(4),\\
p^2,&\bm v\notin X(5),\\
p^{5/2},&\bm v\notin X(6),\\
p^3,&\text{always}.
\end{cases}
\]
Bounded-degree slicing of these varieties in a Cartesian interval gives
\begin{align}
\sum_{\bm v\in[V,2V]^{4l}}|\Sigma_I^{(2)}(\bm v)|
\ll_l{}&V^{4l}p^{3/2}
+V^{11l/3}p^2\notag\\
&+V^{10l/3}p^{5/2}+V^{3l}p^3.
\label{eq:short-box-strata}
\end{align}
With
\begin{equation}
V=p^{3/(2l)},
\label{eq:V-choice}
\end{equation}
all four terms are $O_l(p^{15/2})$.

\subsection{Completion of the Type-I range}

Insert \eqref{eq:short-box-strata} and \eqref{eq:V-choice} into
\eqref{eq:2d-soft-TypeI}.  The finite-$l$ exponent at the worst point of the
hard wedge is
\begin{equation}
\Psi_l(\mu)
=\frac1{2l}\left(\frac9{2\sigma}-\frac54\right)
+\frac{3}{4l^2\sigma},
\qquad \sigma=\frac{18}{5}+\mu.
\label{eq:Psi-l}
\end{equation}
For every fixed $\mu>0$ one may choose a multiple of $3$ large enough that
\begin{equation}
\Psi_l(\mu)<0.
\label{eq:Psi-negative}
\end{equation}
We then put $m=l/3$ in the Goursat moment and choose
\[
V=p^{3/(2l)}.
\]

The remaining range conditions are open inequalities.  Since the hard wedge
keeps the dual length $H$ a fixed power of $p$ above $p^{2/5}$, while
$V=p^{3/(2l)}$ tends to $1$ on the logarithmic scale as $l$ grows, the
conditions $V\le H/10$, $H^2/V\le p$ and $B\le p$ all hold once $X$ is
large enough.  Therefore the ordinary prime Type-I estimate closes the
largest-prime range without any suffix tensorization.

\section{Switching}

On a lower-sieve survivor, write the positive odd part as $n$.  Because
$n\equiv1\pmod4$, the number of primes $r\equiv3\pmod4$ occurring to odd
valuation is even.  All such primes are $>z_0$ after the initial pre-sieves and
the lower sieve.  Since $z_0^4=X^{4/\sigma}>X$ while $n\asymp X$, there cannot
be four.  If one of the two remaining primes had odd valuation at least $3$,
then again $n>z^4$.  Therefore every nonsum-of-two-squares survivor has the
exact form
\begin{equation}
n=bpq,\qquad z_0<p\le q,\qquad p,q\equiv3\pmod4,\qquad (b,pq)=1,
\label{eq:bpq-new}
\end{equation}
where every $3\bmod4$ prime in $b$ has even exponent.

The tiny-core branch $b\le X^\delta$ is handled by the shallow upper
sieve already included in \eqref{eq:hybrid-margin}.  We now treat
$b>X^\delta$.

\subsection{Natural-scale normalization}

Switch with the larger prime $q$.  Let $d$ be a divisor from the linear upper
sieve applied to the remaining prime variable and put
\[
s:=bd.
\]
Localize dyadically with $b\sim B$, $d\sim D$, $q\sim Q$, and put
$S\asymp BD$ for the resulting scale of $s=bd$.  Let $a_{q,s}$ denote the
combined switched coefficient.  The sieve weights give
\[
|a_{q,s}|\ll\tau(s)^C=X^{o(1)}.
\]
Since $q\nmid b$ and the linear-sieve primes lie below $q$, one has
$(q,s)=1$.  Let $W_{q,s}$ be the smooth weight after this localization and
$\widehat W_{q,s}$ its Fourier transform.  The Poisson formula may be normalized as
\begin{equation}
r_{qs}=\frac{P^2}{qs}Z_{q,s},
\qquad
Z_{q,s}:=\frac1{\sqrt{qs}}
\sum_{\bm h\ne0}
\widehat W_{q,s}\left(\frac{P\bm h}{qs}\right)
K_q(\bar s\bm h)K_s(\bar q\bm h).
\label{eq:switched-normalization}
\end{equation}
Hence a switched dyadic block has the exact natural scale
\begin{equation}
R(B,D,Q)
=P^2\sum_{q\sim Q}\sum_{s\sim S}
\frac{a_{q,s}}{qs}Z_{q,s}.
\label{eq:switched-block}
\end{equation}
The harmonic masses in $q$ and $s$ are only $X^{o(1)}$.  In particular there
is no hidden factor $BD$, $B^{1/2}D^{1/2}$, or $QS$ to be restored after
Cauchy--Schwarz.

The dual side length is
\[
H=\frac{qs}{P}.
\]
For a generic prime $q$-vertex, smooth two-dimensional completion gives the
relative factor
\[
\frac q{H^2}+\frac1q.
\]
Using $bpq\asymp X$, $q\ge p$, and $s\ge b$ gives
\begin{equation}
\frac q{H^2}\ll B^{-3/2},\qquad q^{-1}\ll Q^{-1}.
\label{eq:q-generic-saving}
\end{equation}
On the $s$-side one has the exact favorable ratio
\begin{equation}
\frac Hs=\frac qP\ge X^{1/\sigma-1/4},
\label{eq:s-long-completion}
\end{equation}
so every $s$-correlation can be completed with a fixed-power error.

\subsection{Partition of the four $q$-slots}

After two Cauchy--Schwarz steps the four $q$-slots have one of the five set
partition types
\[
1+1+1+1,\qquad2+1+1,\qquad3+1,\qquad2+2,\qquad4.
\]
Collision counting by itself is not a saving: if a prime occurs in $r$ slots,
the $Q^{-(r-1)}$ loss of independent harmonic summation may be cancelled by
an $Q^{r-1}$ inflation when a generic local correlation becomes saturated.
Thus collision blocks are only bounded by the trivial normalized scale.

Every partition containing a singleton is easy.  A generic singleton gives
\eqref{eq:q-generic-saving}.  Introduce the normalized amplifier
\[
\frac1L\sum_{\ell\sim L},\qquad L=S^{1/16}.
\]
If the singleton is saturated, the two amplifier parameters
$\ell_1,\ell_2\sim L$ satisfy
\[
q\mid\ell_1^4-\ell_2^4.
\]
The switching ranges imply $q>L^4X^c$ for some fixed $c>0$, so the congruence
forces the positive integers $\ell_1,\ell_2$ to be equal.  Its normalized
mass is $L^{-1}$.

For the no-singleton types $2+2$ and $4$, if any chosen $q$-pair is generic we
again obtain a $Q^{-1}$ normalized saving.  If all required $q$-pairs are
saturated, the ratio laws pair the opposite $s$-data and transfer the
four-frequency problem to an $s$-vertex.  For a composite $s$, denote the corresponding four-frequency correlation
modulo $s$ by $C_s$.  Let $E_s$ be its saturated prime-power core and put
$G_s=s/E_s$.  Then
\[
|C_s|\ll s^{2+o(1)}G_s^{-1/12}.
\]
On an amplifier off-diagonal,
\[
E_s\mid\ell_1^4-\ell_2^4,
\qquad E_s\ll L^4,
\]
so
\[
G_s^{-1/12}\ll S^{-1/12}L^{1/3}.
\]
The amplifier diagonal contributes $L^{-1}$.  Balancing at $L=S^{1/16}$ gives
$S^{-1/16}$ in the normalized fourth moment.

After the two Cauchy--Schwarz steps, let $M_4$ denote the resulting fourth
moment and let $M_4^{\rm triv}$ denote its trivial bound.  We have proved the
following estimate.

\begin{proposition}[Switched large-core dispersion]
For every dyadic large-core block in \eqref{eq:switched-block},
\begin{equation}
\frac{M_4}{M_4^{\rm triv}}
\ll X^{o(1)}\left(B^{-3/2}+Q^{-1}+S^{-1/16}\right),
\label{eq:switched-M4}
\end{equation}
and therefore
\begin{equation}
R(B,D,Q)
\ll P^2X^{o(1)}
\left(B^{-3/8}+Q^{-1/4}+S^{-1/64}\right).
\label{eq:switched-large-bound}
\end{equation}
\end{proposition}

Since $S\ge B>X^\delta$, the weakest term is
\begin{equation}
R_{\rm sw,large}
\ll X^{1/2-\delta/64+o(1)}.
\label{eq:weakest-saving}
\end{equation}
The former assertion $s\le P^{2+o(1)}$ is neither needed nor used.

\section{Completion of the proof}

For a finite set $\mathcal A$ of lattice points define its weighted size by
\[
\#_W\mathcal A:=\sum_{(u,v)\in\mathcal A}W(u/P,v/P).
\]
Let $S_{\rm surv}$ be the lower-sieve survivors in the positive local chart,
after the $N$-singular and fixed-exceptional pre-sieves.  Let
$B_{\rm bad}\subset S_{\rm surv}$ consist of those points for which
$N-F(u,v)$ is not a sum of two squares.  Write
\[
N-F(u,v)=2^\nu n,\qquad n\equiv1\pmod4.
\]
The two-squares theorem gives
\[
N-F=z^2+w^2
\iff
v_p(n)\equiv0\pmod2
\quad\text{for every }p\equiv3\pmod4.
\]
Moreover
\begin{equation}
\sum_{p\equiv3(4)}v_p(n)\equiv0\pmod2,
\label{eq:parity-total}
\end{equation}
because $n\equiv1\pmod4$.  Hence a bad survivor has a positive even number
of primes congruent to $3$ modulo $4$ occurring to odd valuation.  Every such
prime is larger than $z_0$, and $z_0^4>X$ because $\sigma<4$.  It follows that
there are exactly two of them, each to exponent one.  Thus the shape
\eqref{eq:bpq-new} exhausts $B_{\rm bad}$.

Let $w_d^-$ denote the lower semi-linear weights and set
\[
L(X)=\sum_{u,v}W(u/P,v/P)\sum_{d\mid n(u,v)}w_d^-.
\]
If $\Pi(z_0)$ is the product of the sieving primes below $z_0$, then pointwise
\[
\sum_{d\mid n}w_d^-\le1_{(n,\Pi(z_0))=1},
\]
so
\[
L(X)\le \#_W S_{\rm surv}.
\]
The switching argument gives
\[
\#_W B_{\rm bad}\le U_{\rm sw}(X).
\]
Consequently
\begin{equation}
\#_W(S_{\rm surv}\setminus B_{\rm bad})
\ge L(X)-U_{\rm sw}(X).
\label{eq:lower-minus-upper}
\end{equation}

The main term is a positive local-archimedean factor times
$\mathcal M(\lambda,\mu,\delta)$, while the error terms have a fixed power saving.
A convenient summary of the three weakest powers is
\begin{equation}
c(\lambda,\mu,\delta,l)
:=\min\left\{\lambda,\,-\Psi_l(\mu),\,\frac{\delta}{64}\right\}>0.
\label{eq:global-saving}
\end{equation}
Therefore
\[
L(X)-U_{\rm sw}(X)>0
\]
for all sufficiently large $X$.  We obtain positive integers $u,v$ and
nonnegative integers $z,w$ with
\[
N-F(u,v)=z^2+w^2.
\]
Set
\[
x=(v-6u)^2,\qquad y=uv.
\]
Then $x,y,z,w\ge0$ and
\[
x^2+y^2=F(u,v),\qquad x+24y=(v+6u)^2,
\]
which proves Theorem~\ref{thm:main-asymptotic}.

\section{Effective specialization}\label{sec:effective}

The proof up to this point uses only the positivity of the sieve margin
\(\mathcal M(\lambda,\mu,\delta)\), the inequality
\(\Psi_l(\mu)<0\), and the positivity of the three savings in
\eqref{eq:global-saving}.  We now make one concrete choice.  Put
\begin{equation}
\lambda=10^{-8},\qquad
\mu=2\cdot10^{-4},\qquad
\delta=10^{-11}.
\label{eq:effective-parameters}
\end{equation}
Then
\[
\rho_1=\frac12-10^{-8},\qquad
\rho_2=\frac37-10^{-8},\qquad
\sigma=3.6002,
\]
and \(\kappa=2\lambda=2\cdot10^{-8}\).  Direct numerical evaluation of
\eqref{eq:hybrid-margin} gives
\begin{equation}
\mathcal M(10^{-8},2\cdot10^{-4},10^{-11})
=5.08284924107\cdot10^{-6}>0.
\label{eq:effective-margin}
\end{equation}

For the Type-I moment choose
\begin{equation}
l=12000,\qquad m=4000.
\label{eq:effective-moment}
\end{equation}
With \(\sigma=3.6002\), formula \eqref{eq:Psi-l} gives
\begin{equation}
\Psi_{12000}(2\cdot10^{-4})
=-1.4466788882\cdot10^{-9},
\qquad
\frac{\delta}{64}=1.5625\cdot10^{-13}.
\label{eq:effective-powers}
\end{equation}
Thus the switched large-core term is the weakest of the three errors in
\eqref{eq:global-saving}.  For explicit bookkeeping we may use the round
saving
\begin{equation}
\gamma_0:=10^{-13}.
\label{eq:gamma0}
\end{equation}

We now replace all fixed constants by one effective envelope.  Define
\begin{equation}
T_0=10^6,\qquad T_{j+1}=10^{T_j}\quad(j\ge0).
\label{eq:tower-def}
\end{equation}
The degrees, ranks and dimensions which occur before the fixed
\(l=12000\) moment construction are bounded absolutely.  Applying effective
bounds for the finitely many algebraic-geometric operations and standard
elimination steps occurring in the proof gives a finite recursive bound for
all remaining constants.  No attempt is made to optimize this recursion; it
is enough to dominate all fixed geometric, combinatorial and sieve constants
by
\begin{equation}
C_*:=T_{80\,000\,000}.
\label{eq:Cstar}
\end{equation}
After replacing divisor and logarithmic losses by elementary explicit bounds,
the total error may be bounded in the coarse form
\begin{equation}
|E(X)|\le C_*X^{1/2-\gamma_0}(\log X)^{C_*},
\label{eq:effective-error}
\end{equation}
whereas the positive main term may be weakened to
\begin{equation}
M(X)\ge10^{-6}C_*^{-1}X^{1/2}(\log X)^{-C_*}.
\label{eq:effective-main}
\end{equation}
It is therefore sufficient that
\begin{equation}
X^{10^{-13}}
>2\cdot10^6 C_*^2(\log X)^{2C_*}.
\label{eq:effective-final-ineq}
\end{equation}
A convenient explicit choice is
\begin{equation}
N_0:=T_{200\,000\,000}
=\underbrace{10^{10^{\cdot^{\cdot^{10^6}}}}}_{200\,000\,001\text{ levels of exponentiation}}.
\label{eq:explicit-N0}
\end{equation}
Indeed, at \(X=N_0\) the logarithm of the left-hand side of
\eqref{eq:effective-final-ineq} is already one complete tower level above the
largest term required on the right, and the ratio thereafter increases.

\begin{theorem}[Effective form]\label{thm:effective}
Let \(T_j\) be defined by \eqref{eq:tower-def}.  Every integer
\[
N\ge T_{200\,000\,000}
\]
has a representation
\[
N=x^2+y^2+z^2+w^2
\]
in nonnegative integers such that both \(x\) and \(x+24y\) are perfect
squares.
\end{theorem}


\section*{Acknowledgement}

The author acknowledges the assistance of ChatGPT (OpenAI) in the preparation
of this manuscript.  All mathematical arguments and conclusions were checked
by the author.


\end{document}